\documentclass[11pt]{amsart}
\usepackage{amsmath,amssymb,mathrsfs}
\usepackage{pstricks,pst-plot}
\usepackage[small,heads=LaTeX,nohug]{diagrams}
\usepackage[colorlinks=true,linkcolor=black,citecolor=black,urlcolor=black]{hyperref}

\usepackage{pst-3d}

\psset{unit=1pt}
\psset{arrowsize=4pt 1}
\psset{linewidth=.5pt}

\renewcommand{\theenumi}{\roman{enumi}}

\newcommand{\define}{\textbf}

\newcommand{\excise}[1]{}

\newcommand{\ul}{\underline}

\newcommand{\isom}{\cong}
\renewcommand{\setminus}{\smallsetminus}
\renewcommand{\phi}{\varphi}

\renewcommand{\tilde}{\widetilde}
\renewcommand{\hat}{\widehat}
\renewcommand{\bar}{\overline}

\newcommand{\N}{\mathbb{N}}

\newcommand{\R}{\mathbb{R}}
\newcommand{\Z}{\mathbb{Z}}
\renewcommand{\P}{\mathbb{P}}
\newcommand{\A}{\mathbb{A}}

\renewcommand{\O}{\mathcal{O}}
\newcommand{\RR}{\mathcal{R}}

\newcommand{\XX}{\mathfrak{X}}

\newcommand{\cC}{\mathcal{C}}
\newcommand{\cD}{\mathcal{D}}

\newcommand{\cL}{\mathcal{L}}
\newcommand{\cM}{\mathcal{M}}

\newcommand{\mx}{\mathrm{max}}

\newcommand{\kk}{k}

\newcommand{\ba}{\mathbf{a}}
\newcommand{\bb}{\mathbf{b}}
\newcommand{\bc}{\mathbf{c}}
\newcommand{\bm}{\mathbf{m}}
\newcommand{\bV}{\mathbf{V}}

\newcommand{\id}{\mathrm{id}}

\DeclareMathOperator{\Sym}{Sym}

\DeclareMathOperator{\gr}{gr}

\DeclareMathOperator{\ord}{ord}
\DeclareMathOperator{\Pic}{Pic}
\DeclareMathOperator{\Proj}{Proj}
\DeclareMathOperator{\Spec}{Spec}

\DeclareMathOperator{\Bs}{Bs}
\DeclareMathOperator{\Span}{Span}
\DeclareMathOperator{\lcm}{lcm}

\newcommand{\Cone}{\mathrm{Cone}}
\newcommand{\Conv}{\mathrm{Conv}}

\newcommand{\Eff}{\bar{\mathrm{Eff}}}

\newtheorem{theorem}{Theorem}[section]
\newtheorem{lemma}[theorem]{Lemma}
\newtheorem{proposition}[theorem]{Proposition}
\newtheorem{corollary}[theorem]{Corollary}

\theoremstyle{definition}
\newtheorem{definition}[theorem]{Definition}
\newtheorem{remark}[theorem]{Remark}
\newtheorem{example}[theorem]{Example}

\begin{document}

\title{Okounkov bodies and toric degenerations}
\author{Dave Anderson}
\address{Department of Mathematics\\University of Washington\\Seattle, WA 98195, USA}
\email{dandersn@math.washington.edu}

\address{FSMP--Institut de Math\'ematiques de Jussieu, 75013
Paris, France}
\email{andersond@math.jussieu.fr}

\keywords{Okounkov body, toric variety, degeneration}
\date{October 9, 2012}
\thanks{The author was partially supported by NSF Grants DMS-0502170 and DMS-0902967, and also by the Clay Mathematics Institute as a Liftoff Fellow.}


\begin{abstract}
Let $\Delta$ be the Okounkov body of a divisor $D$ on a projective variety $X$.  We describe a geometric criterion for $\Delta$ to be a lattice polytope, and show that in this situation $X$ admits a flat degeneration to the corresponding toric variety.  This degeneration is functorial in an appropriate sense.
\end{abstract}

\maketitle

\section{Introduction}

Let $X$ be a projective algebraic variety of dimension $d$ over an algebraically closed field $\kk$, and let $D$ be a big divisor on $X$.  (Following \cite{lm}, all divisors are Cartier in this article.)  As part of his proof of the log-concavity of the multiplicity function for representations of a reductive group, Okounkov showed how to associate to $D$ a convex body $\Delta_{Y_\bullet}(D) \subseteq \R^d$ \cite{ok1,okounkov}.  The construction depends on a choice of \emph{flag} $Y_\bullet$ of subvarieties of $X$, that is, a chain $X=Y_0 \supset Y_1 \supset \cdots \supset Y_d$, where $Y_r$ is a subvariety of codimension $r$ in $X$ which is nonsingular at the point $Y_d$.  One uses the flag to define a valuation $\nu=\nu_{Y_\bullet}$, which in turn defines a graded semigroup $\Gamma_{Y_\bullet} \subseteq \N\times\N^d$; the convex body $\Delta=\Delta_{Y_\bullet}(D)$ is the intersection of $\{1\}\times \R^d$ with the closure of the convex hull of $\Gamma_{Y_\bullet}$ in $\R\times\R^d$.  The details of this construction will be reviewed in \S\ref{s:ok}.

Recently, Kaveh-Khovanskii \cite{kk} and Lazarsfeld-Musta{\c{t}}{\u{a}} \cite{lm} have systematically developed this construction, and exploited it to show that $\Delta_{Y_\bullet}(D)$ captures much of the geometry of $D$.  For example, the volume of $D$ (as a divisor) is equal to the Euclidean volume of $\Delta$ (up to a normalizing factor of $d!$), and one can use this to prove continuity of the volume function, as a map $N^1(X)_\R \to \R$ (see \cite[Theorem B]{lm} and the references given there).  Many intersection-theoretic notions can also be defined and generalized using the convex bodies $\Delta(D)$; this is discussed at length in \cite{kk}.

These \emph{Okounkov bodies}---as $\Delta_{Y_\bullet}(D)$ is called in the literature stemming from \cite{lm}---are generally quite difficult to compute.  They are often not polyhedral; when polyhedral, they are often not rational; and even if $\Delta$ is a rational polyhedron, the semigroup used to define it need not be finitely generated.  In fact, the numbers appearing as volumes or coordinates of Okounkov bodies can be quite general \cite{klm}; they certainly may be irrational.

One situation, however, is easy.  When $X$ is a smooth toric variety, $D$ is a $T$-invariant ample divisor, and $Y_\bullet$ is a flag of $T$-invariant subvarieties, the Okounkov body $\Delta_{Y_\bullet}(D)$ is the lattice polytope associated to $D$ by the usual correspondence of toric geometry \cite[Proposition 6.1]{lm}.

In this article, we extend the connection between Okounkov bodies and toric varieties.  Suppose $X$ is an arbitrary variety and $D$ an ample divisor such that the corresponding semigroup $\Gamma_{Y_\bullet}$ is finitely generated, so $\Delta_{Y_\bullet}(D)$ is a rational polytope.  We construct a flat degeneration from $X$ to the (not necessarily normal) toric variety defined by $\Gamma_{Y_\bullet}$.  The normalization of the limit is the toric variety corresponding to the polytope $\Delta_{Y_\bullet}(D)$ (see Theorem \ref{t:deg}).

More generally, for a big divisor $D$ and a linear system $V \subseteq H^0(X,\cL)$, one has a semigroup $\Gamma_{Y_\bullet}(V)$ and an Okounkov body $\Delta_{Y_\bullet}(V)$ (see \S\ref{s:ok} for the precise definitions).  Write $X(V)$ for the closure of the image of $X$ in $\P(V)$.  Abusing notation slightly, let $\nu(V)\subseteq \Z^d$ denote the image of $V\setminus\{0\}$ under the valuation $\nu$.  A variety $X$ \emph{admits a flat degeneration} to a variety $X_0$ if there is a flat family $\XX \to \A^1$ such that the fiber $X_t$ is isomorphic to $X$ for $t\neq 0$, and the fiber over $0$ is $X_0$.

\begin{theorem}\label{t:main1}
Let $\nu=\nu_{Y_\bullet}$ be the valuation associated to a flag of subvarieties of $X$, and let $V\subseteq H^0(X,\cL)$ be a linear system such that $\Gamma=\Gamma_{Y_\bullet}(V)$ is finitely generated.
\renewcommand{\theenumi}{\alph{enumi}}
\begin{enumerate}
\item The variety $X(V)$ admits a flat degeneration to the (not necessarily normal) toric variety $X_0=\Proj \kk[\Gamma]$.  
The normalization of $X_0$ is the (normal) toric variety corresponding to the polytope $\Delta_{Y_\bullet}(V)$.

\medskip
\item If a torus $T$ acts on $X$, such that $V$ is a $T$-invariant linear system and $Y_\bullet$ consists of $T$-invariant subvarieties, then the degeneration is $T$-equivariant.

\medskip
\item Suppose $V'\subseteq V$ is a subsystem inducing a birational morphism $\phi\colon X(V) \to X(V')$, and whose semigroup is also finitely generated.  The corresponding degenerations of $X(V)$ and $X(V')$ are compatible: there is a commuting diagram of flat families
\begin{diagram}
\XX & & \rTo^{\Phi} & & \XX' \\
        & \rdTo &    & \ldTo \\
        &      &  \A^1,
\end{diagram}
such that $\Phi_t \isom \phi$ for $t\neq 0$, and $\Phi_0 = \phi_0$.

\medskip
\item Fix $0\leq r\leq d$, and assume $Y_r$ does not lie in the base locus of $V$.  Then the subvariety $Y_r(V) \subseteq X(V)$ admits a flat degeneration to a toric variety, compatible with that of $X(V)$.  The toric limit $(Y_r)_0$ corresponds to the face $\Delta_{Y_\bullet}(V) \cap ( \{0\}^{d-r}\times \Z^r )$.
\end{enumerate}

\end{theorem}

\noindent
The theorem summarizes the results in Theorem~\ref{t:deg} and Proposition~\ref{p:compatible}, together with their applications in Examples~\ref{ex:torus}, \ref{ex:birational}, and \ref{ex:subvar}.

In general, the question of when the semigroup $\Gamma_{Y_\bullet}(V)$ is finitely generated appears to be quite subtle, sensitive to both the linear system $V$ and the flag $Y_\bullet$ (see Example~\ref{ex:1} and \S\ref{ss:counterex}).  It is closely related to the existence of a finite SAGBI basis for the ring $R(V)$ (see \cite[\S5.6]{kk-alg}).  Motivated by this, we provide some criteria for finite generation in \S\ref{s:conditions}.  The key notion introduced there is that of a \emph{maximal divisor} in a linear system $V$, with respect to a fixed divisor $Y$.  Such divisors have maximal order along $Y$, and also have the property that their multiples are maximal in the powers $V^m$; thus they bound the growth of the semigroup $\Gamma_{Y_\bullet}(V)$.

The main purpose of this article is to develop a formal framework for the toric degenerations resulting from a finitely generated semigroup (\S\ref{s:deg}), and to show how criteria for finite generation can be applied in examples (\S\ref{s:examples}).  We also hope to motivate further study of the finite generation problem.

This work began with a desire to understand toric degenerations of flag varieties and Schubert varieties, which have been described from various points of view over the last fifteen years (see, e.g., \cite{gl,caldero,ab,kaveh,km}).  I plan to return to this subject in future paper \cite{and}, generalizing the example given in \S\ref{ss:bs}.

\smallskip
\noindent
{\it Acknowledgements.}  I am grateful to Sara Billey, Jos\'e Luis Gonz\'alez, S\'andor Kov\'acs, Rob Lazarsfeld, Ezra Miller, Mircea Musta{\c{t}}{\u{a}}, and Rekha Thomas for useful comments and fruitful discussions, and especially to Megumi Harada and Kiumars Kaveh for several valuable suggestions.  I also thank Bill Fulton and Ben Howard for helping me learn about toric degenerations.

\section{Valuations and semigroups}\label{s:vals}

Throughout this article, let $k$ denote an algebraically closed field.  Let $K$ be a field extension of $\kk$; in our applications, $K$ will be the function field of a variety over $k$, and we may always assume it has finite transcendence degree.  Equip $\Z^d$ with the lexicographic order, making it an ordered abelian group.  Following \cite{kk}, a $\Z^d$-\define{valuation} on $K$ is a map $\nu:K\setminus\{0\} \to \Z^d$ satisfying
\begin{enumerate}
\item for all $f,g\in K$, $\nu(f+g)\geq\min\{\nu(f),\nu(g)\}$ (when all these elements are nonzero); and
\item for nonzero $f,g$ in $K$, $\nu(fg)=\nu(f)+\nu(g)$.
\end{enumerate}
It follows from the first of these conditions that
\begin{equation}\label{e:valcond}
\nu(f+g) > \min\{\nu(f),\nu(g)\} \text{ only if } \nu(f)=\nu(g).
\end{equation}
(Indeed, suppose $\nu(f)<\nu(g)$, but $\nu(f+g)>\nu(f)$.  Then $\nu(f)=\nu((f+g)+(-g)) \geq \min\{\nu(f+g),\nu(g)\} >\nu(f)$, a contradiction.)

To any (nonzero) finite-dimensional $\kk$-vector subspace $V\subseteq K$, we associate a ring, a semigroup, a cone, and a convex body, still following \cite{kk}.  To wit,
\begin{itemize}
\item the ring $R(V)$ is the graded ring $\bigoplus_{m\geq 0} V^m$, where $V^m$ is the subspace spanned by products $f_1\cdots f_m$, with each $f_i\in V$;

\smallskip

\item $\Gamma(V)=\Gamma_\nu(V) = \{(m,\nu(f))\in \N\times\Z^d \,|\, f\in V^m\setminus\{0\}\}$, a graded semigroup in $\N\times \Z^d$;

\smallskip

\item $\Cone(V) = \Cone(\Gamma(V))$ is the closure of the convex hull of $\Gamma(V)$ in $\R\times \R^d$;

\smallskip

\item the \define{Newton-Okounkov body} is $\Delta(V) = \Delta_\nu(V) = \Cone(V) \cap (\{1\}\times \R^d)$.
\end{itemize}

It will be convenient to use an ``extended'' valuation 
\[
  \hat\nu:R(V)\setminus\{0\} \to \N\times \Z^d,
\]
defined by $\hat\nu(f) = (m,\nu(f_m))$, where $f_m$ is the highest-degree 
homogeneous component of $f$.  By definition, the image of $\hat\nu$ is $\Gamma=\Gamma_\nu(V)$.   Also, we will use a slightly modified total order on $\Gamma$, induced from the order on $\N\times\Z^d$ given by
$(m_1,u_1) \leq (m_2,u_2)$ iff $m_1<m_2$ or $m_1=m_2$ and $u_1\geq u_2$.  Note the switch: this is the opposite of the order defined in \cite[p.~23]{kk}.

We record a few easy facts which will be useful later:

\begin{lemma}\label{l:easy}
\renewcommand{\theenumi}{\alph{enumi}}
Let $\Gamma=\Gamma_\nu(V)$ and $R=R(V)$.
\begin{enumerate}
\item The order on $\N\times\Z^d$ respects addition, and induces an ordered group structure on $\Z\times\Z^d$.

\smallskip

\item Given any $(m,u)\in\N\times\Z^d$, the set 
\[
  \Gamma_{\leq (m,u)} = \{(m',u')\in \Gamma \,|\, (m',u')\leq (m,u)\}
\]
is finite.

\smallskip

\item For any $(m,u)\in \Gamma$,
\[
  R_{\leq (m,u)} = \{ f\in R \,|\, \hat\nu(f)\leq (m,u) \}
\]
is a (finite-dimensional) vector subspace of $R$.  (I.e., it is closed under addition.)

\smallskip

\item For any $(m,u)$ and $(m',u')$ in $\Gamma$, we have
\[
  R_{\leq (m,u)} \cdot R_{\leq (m',u')} \subseteq R_{\leq (m+m',u+u')}.
\]
\qed
\end{enumerate}
\end{lemma}

\begin{lemma}\label{l:fg}
Let $\Gamma\subseteq \N\times\N^d$ be a graded semigroup, and suppose $\Cone(\Gamma)$ is generated by $\Gamma \cap (\{1\}\times \N^d)$.  Then $\Gamma$ is finitely generated. \qed
\end{lemma}

\noindent
This is a special case of \cite[Corollary~2.10]{bg}; I thank Rekha Thomas for this reference (and for showing me an alternative proof).

We need one more algebraic fact:

\begin{lemma}\label{l:normal}
Let $\Gamma\subseteq \N\times\N^d$ be as in Lemma~\ref{l:fg}.  Let $\Delta \subseteq \R^d$ be the polytope given by $\{1\}\times\Delta = \Cone(\Gamma)\cap (\{1\}\times\R^d)$, and assume that every lattice point in $d\Delta$ lies in $\Gamma$, i.e., $\Gamma\cap(\{d\}\times\N^d) = (\{d\}\times d\Delta)\cap(\{d\}\times\N^d)$.  Then $\Proj\kk[\Gamma]$ is the normal toric variety corresponding to $\Delta$. \qed
\end{lemma}

The content of the statement is that $\Proj\kk[\Gamma]$ is normal, and this is a consequence of fact that the semigroup ring $\kk[\Gamma_d]$ is normal, where $\Gamma_d\subseteq\Gamma$ is the sub-semigroup generated by $\Gamma\cap(\{d\}\times\N^d)$ and the polytope $\Delta$ is $d$-dimensional (see, e.g., \cite[Corollary~7.45]{bg} for a finer criterion).  It follows that $\Proj\kk[\Gamma]\isom\Proj\kk[\Gamma_d]$.

\begin{remark}\label{rmk:proj-norm}
While $\Proj\kk[\Gamma]$ is normal in the situation of Lemma~\ref{l:normal}, it need not be projectively normal---that is, the semigroup $\Gamma$ may not contain all the lattice points of $\Cone(\Gamma)$.  One has normality of the graded ring $\kk[\Gamma]$ when $p(\Delta\cap\Z^d) = (p\Delta\cap\Z^d)$ for all $p>0$, but this is far from true in general.  See \cite{bg}, especially Chapter~2, for discussions of non-normal semigroup rings.

Another normality criterion, with slightly weaker hypotheses than those of Lemma~\ref{l:normal}, is given in \cite[Corollary~3.7]{ck}.
\end{remark}

\begin{remark}\label{rmk:lex-order}
It may be interesting to consider valuations with respect to different total orders on the lattice $\Z^d$; one would need a variant of Lemma~\ref{l:caldero} in order for the most general results of this paper to go through (e.g., Proposition~\ref{p:ring-deg}).  However, for valuations arising from the Okounkov construction (to be described in the next section), lexicographic order is the most natural one to use.  Since our emphasis in the 
present paper is on such valuations, we shall work with lexicographic order exclusively.
\end{remark}

\section{Okounkov bodies}\label{s:ok}

We will be mainly interested in valuations which arise from a geometric 
construction; see \cite{lm} for a more detailed description.  Let $X$ be a 
projective variety of dimension $d$, and let $Y_\bullet$ be a flag of 
subvarieties; recall that this means each $Y_i$ has codimension $i$, and is nonsingular at the point $Y_d$.  (In \cite{lm}, the condition that $Y_d$ be a nonsingular point defines an \emph{admissible flag}, but we will not consider 
flags without this condition.)  Let $\cL$ be a line bundle on $X$, and let $V\subseteq H^0(X,\cL)$ be a linear system.  For each $m\geq 0$, let $V^m$ be the image of $\Sym^m V$ under the natural map to $H^0(X,\cL^{\otimes m})$.  Thus $R(V) = \bigoplus V^m$ is a subring of the section ring $\bigoplus H^0(X,\cL^{\otimes m})$.  

The valuation $\nu=\nu_{Y_\bullet}$ is defined as follows.  For $f\in V^m$, set $\nu_1 = \nu_1(f) = \ord_{Y_1}(f)$.  Let $y_1$ be a local equation for $Y_1$, and let $f_1$ be the restriction of $f\cdot y_1^{-\nu_1}$ to $Y_1$.  Now set $\nu_2 = \nu_2(f) = \ord_{Y_2}(f_1)$, and repeat, defining $\nu_r$ similarly for each $Y_r$.

\begin{remark}
Appropriately identifying $V$ with a subspace of the field $K=K(X)$ of rational functions on $X$, $\nu$ extends to a valuation on $K$ in the sense of \S\ref{s:vals}.
\end{remark}

When the constructions of \S\ref{s:vals} are applied to this situation, with 
$\nu=\nu_{Y_\bullet}$, we will often write $\Gamma_\nu(V)=\Gamma_{Y_\bullet}(V)$ and $\Delta_\nu(V) = \Delta_{Y_\bullet}(V)$, and call the latter the \define{Okounkov body} associated to $V$.  In this sense, Okounkov bodies 
are special cases of Newton-Okounkov convex bodies.

The following fact will play a key role in the toric degenerations constructed in \S\ref{s:deg}:

\begin{lemma}[{\cite[\S2.3]{ok1}}, {\cite[Lemma 1.3]{lm}}]\label{l:dim1}
Let $V\subseteq H^0(X,\cL)$ be any subspace, and let $Y_\bullet$ be a 
(complete) flag, inducing a valuation $\nu_{Y_\bullet}$.  For $u\in \Z^d$, set 
$V_{\geq u} = \{ f \in V \,|\, \nu_{Y_\bullet}(f)\geq u \}$, and define 
$V_{> u}$ similarly.  Then
\[
  \dim (V_{\geq u}/V_{> u} ) \leq 1.
\]
\qed
\end{lemma}

In this lemma, $X$ need not be projective, and $V$ need not be finite-dimensional.  However, the claim is particular to valuations arising from complete flags $Y_\bullet$; the analogous statement does not hold for the zero valuation, for example, or for any valuation defined by an incomplete flag \cite[Remark 1.4]{lm}.

\begin{remark}
In \cite{lm}, arbitrary \emph{graded linear series} are considered.  These are spaces of sections $V_\bullet$, with $V_m \subseteq H^0(X,\cL^{\otimes m})$, satisfying $V_\ell \cdot V_m \subseteq V_{\ell+m}$.  In this generality, one can construct any Newton-Okounkov body as an Okounkov body for some graded linear series, so the two perspectives are essentially equivalent, at least for valuations arising from complete flags.  Here, however, we will only use linear series which are 
``generated in degree $1$''---that is, $V_m = V^m$.
\end{remark}

\section{Maximal divisors and finite generation}\label{s:conditions}

Continuing the notation of \S\ref{s:ok}, let $Y_\bullet$ be a flag in $X$, and let $\nu=\nu_{Y_\bullet}$ be the associated valuation.  Here we give several criteria for the semigroup $\Gamma_\nu(V)$ to be finitely generated.  We first set up some further notation for use in this section.

For $0\leq r\leq d$, let $p_r:\Z^d \to \Z^r$ be the projection on the first $r$ coordinates, and let $\nu|_r = (\nu_1,\ldots,\nu_r)$ be the composition of $\nu$ with $p_r$.  For a linear system $V$, let
\[
  \Gamma_\nu|_r(V) = \{ (m,\nu|_r(f)) \,|\, f\in V^m \} = (\id\times p_r)(\Gamma_\nu(V)) \subset \Z \times \Z^r.
\]

We will also need some notation for certain restricted linear systems, using a construction from \cite{jow}.  Given an effective divisor $E$, consider the map $H^0(X, \cL\otimes\O(-E)) \to H^0(X,\cL)$ arising from the section of $\O(E)$ defining $E$.  Write $V-E \subseteq H^0(X,\cL\otimes\O(-E))$ for the inverse image of $V$ under this map.

Now fix a flag $Y_\bullet$.  Given an $r$-tuple of integers $\ba=(a_1,\ldots,a_r)$, define
\[
  V(\ba) \subseteq H^0(Y_r,\cL\otimes\O(-a_1 Y_1 - \cdots - a_r Y_r)|_{Y_r})
\]
inductively as the image of $V(a_1,\ldots,a_{r-1})-a_r Y_r \subseteq H^0(Y_{r-1},\cL\otimes\O(-a_1 Y_1 - \cdots - a_{r-1} Y_{r-1})|_{Y_{r-1}}\otimes \O(-a_r Y_r))$ under the restriction map from $Y_{r-1}$ to $Y_r$.

Finally, a linear system $V$ is said to be \emph{saturated} with respect to a divisor $Y$ if the set $\{ \ord_Y(f) \,|\, f\in V \}$ is a complete interval, i.e., it is equal to $[a_{\mathrm{min}},a_{\mathrm{max}}] \subset \Z$.

\medskip
Our first criterion for finite generation is simple, but useful; its proof is immediate from Lemma~\ref{l:fg}.

\begin{proposition}\label{p:fg-easy}
Suppose the Okounkov body $\Delta_{Y_\bullet}(V)$ is equal to the convex hull of the lattice points $\nu(V) \subseteq \Z^d$.  Then $\Gamma_\nu(V)$ is finitely generated. \qed
\end{proposition}

\begin{example}\label{ex:easy-curve}
Let $X$ be a smooth projective curve of genus $g$, with the flag $Y_\bullet$ given by $X\ni P$ for a point $P$.  Let $\cL$ be a line bundle of degree $n\geq 2g+1$.  Then the Okounkov body $\Delta_{Y_\bullet}(V)$ is the interval $[0,n]$, so the hypothesis of Proposition~\ref{p:fg-easy} holds (for the complete linear system) exactly when there is a section vanishing to order $n$ at $P$, that is, when $\cL\isom \O(nP)$.

In fact, a partial converse to the proposition holds in this case: $\Gamma_\nu(V)$ is (finitely) generated by elements of degree $1$ only if $\cL\isom\O(nP)$.
\end{example}

To generalize the situation of Example~\ref{ex:easy-curve} to higher dimensions, we introduce some terminology.  Recall that the \emph{cone of pseudoeffective divisors} on a variety $X$ is the closed convex hull of all numerical classes of effective divisors in the N\'eron-Severi space $N^1(X)_\R$ (see \cite[\S2]{rlaz}).  A \emph{facet} of a convex cone is the intersection of the cone with a supporting hyperplane.

\begin{definition}\label{Dmax}
Let $V\subseteq H^0(X,\cL)$ be a linear system, and fix a prime divisor $Y\subseteq X$.  Consider an effective divisor $aY + b_1B_1 + \cdots + b_kB_k + e_1E_1+\cdots+e_\ell E_\ell$ in $V$, written as a sum of irreducible components with multiplicities.  The divisor is \define{maximal (in $V$ with respect to $Y$)} if each component $b_iB_i$ is contained in the base scheme of $V$, and the classes of the $E_i$ lie in a facet of the pseudoeffective cone $\Eff(X)$ not containing the class of $Y$.  
\end{definition}

\begin{example}\label{ex:1}
Let $X$ be an elliptic curve, let $P\in X$ be a point, and let $V=H^0(X,\O(3P))$.  Then $P$ is maximal with respect to $P'$ if and only if $P'$ is an inflection point for the embedding by $3P$, i.e., if $\O(3P') \isom \O(3P)$.  (Otherwise, $\O(3P)\isom \O(2P' + Q')$ for some point $Q'\neq P'$, and $2P' + Q'$ is not maximal, since $P'\equiv Q'$ in $N^1(X)$.)  On the other hand, let $P'$ be a general point, and suppose $V\subseteq H^0(X,\O(3P))=H^0(X,\O(2P'+Q'))$ is the subspace of sections vanishing at $Q'$.  Then $2P'+Q'$ is maximal in $V$ with respect to $P'$.
\end{example}

\begin{example}\label{ex:2}
Let $X = C\times C$ be a product of elliptic curves.  Choose points $p_1,p_2\in C$, and write $C_1 = \{p_1\}\times C$ and $C_2 = C\times \{p_2\}$.  The divisor $3C_1+3C_2$ is maximal in $V=H^0(X,\O(3C_1+3C_2))$ with respect to $C_1$ or $C_2$.  (Here the cone $\Eff(X)$ is equal to the nef cone.  This is a round cone bounded by classes $\alpha$ such that $(\alpha^2)=0$, so the facets are just the extremal rays.  Since $(C_2^2)=0$, its class lies on a facet not containing $C_1$.)

On the other hand, let $V=H^0(X,\O(3C_1+3C_2))$ as before, and let $\delta \subseteq X$ be the diagonal.  One can find divisors $\delta + D$ in $V$, for some effective $D$, but these are not maximal with respect to $\delta$.  (Indeed, $D\equiv 3C_1+3C_2 - \delta$ in $N^1(X)$, so $(D^2)=3>0$ and therefore the components of $D$ do not lie on a single facet of $\Eff(X)$.)  In the linear system $V$, there are no effective divisors of order greater than $1$ along $\delta$, since $2\delta + D \equiv 3C_1+3C_2$ implies $(D^2) = -3 < 0$.  Thus there are no maximal divisors in $V$ with respect to $\delta$.
\end{example}

We now record some basic facts about maximal divisors.  The nomenclature is justified by the first lemma:

\begin{lemma}\label{l:maxl}
Let $D = a_1D_1 + \cdots + a_k D_k$ be an effective divisor with irreducible components $D_i$, which is maximal in $V\subseteq H^0(X,D)$ with respect to $D_1$; and let $t\in V$ be any section, defining a linearly equivalent effective divisor $b_1 D_1 + E$.  Then $b_1\leq a_1$.
\end{lemma}

\begin{proof}
Suppose $D_2,\ldots,D_j$ are the components of $D$ with $a_iD_i \subseteq \Bs(V)$.  Since $a_2 D_2 + \cdots + a_j D_j \subseteq \Bs(V)$, we can write $E=b_2D_2 + \cdots + b_jD_j + E'$, with $b_i\geq a_i$.  The effective divisor $E'$ is linearly equivalent to $(a_1-b_1)D_1 + (a_2-b_2)D_2 + \cdots + (a_j-b_j)D_j + a_{j+1}D_{j+1} + \cdots + a_kD_k$.  The claim follows from Lemma~\ref{l:pointed-cone} below, using the assumption that $D$ is maximal and the fact that $\Eff(X)$ is a pointed cone \cite[Lemma~4.6]{lm}.
\end{proof}

\begin{lemma}\label{l:pointed-cone}
Let $\sigma$ be a pointed convex cone in a (real or rational) vector space.  Suppose $v =cu + \sum d_i v_i + \sum e_j w_j$ lies in $\sigma$, along with the vectors $u$, $v_i$, and $w_j$.  If the $w_j$ lie in a facet of $\sigma$ not containing $u$, and all $d_i\leq 0$, then $c\geq 0$.
\end{lemma}

\begin{proof}
Take a functional $\phi$ defining a hyperplane supporting the facet containing the $w_j$, so $\sigma \subseteq \{ \phi \geq 0\}$ and $\phi(w_i)=0$ for all $j$.  Then $\phi(v) = c\phi(u) + \sum d_i \phi(v_i) \geq 0$ implies $c\geq 0$.
\end{proof}

The next lemma is an immediate consequence of the additivity properties of base loci:

\begin{lemma}\label{l:max-sum}
Fix prime divisors $D_1,\ldots,D_k$.  Suppose $D=a_1D_1 + \cdots + a_k D_k$ is maximal in $V\subseteq H^0(X,D)$ and $D'=a'_1D_1 + \cdots + a'_k D_k$ is maximal in $V'\subseteq H^0(X,D')$, both with respect to $D_1$.  Then $D+D'$ is maximal in $V\cdot V'\subseteq H^0(X,D+D')$ with respect to $D_1$.  In particular, for any section $t\in V\cdot V'$, we have $\ord_{D_1}(t)\leq a_1+a'_1$. \qed
\end{lemma}

Now fix a flag $Y_\bullet$.  Suppose $V$ contains a maximal divisor $D$ with respect to $Y_1$, and also contains a section not vanishing along $Y_1$.  It follows immediately from Lemmas~\ref{l:maxl} and \ref{l:max-sum}, together with Lemma~\ref{l:fg}, that the semigroup $\Gamma_1:=\Gamma_\nu|_1(V)$ is finitely generated.  Indeed, $\Cone(\Gamma_1) \subseteq \R \times \R$ is generated by $(1,0)$ and $(1,a)$, where $a$ is the coefficient of $Y_1$ in the maximal divisor $D$.

In order to extend this to the semigroups $\Gamma_v|_r(V)$ for $r>1$, we define some conditions on (restricted) linear systems.

\begin{definition}
Given two linear systems $V \subseteq H^0(X,\cL)$ and $W\subseteq H^0(X,\cM)$, and an integer $r$ with $1\leq r\leq d$, consider the following conditions:
\begin{enumerate}
\item For each $\ba \in \nu|_{r-1}(V)$, $V(\ba)$ contains a maximal divisor, as well as a section that does not vanish at the point $Y_d$; and similarly for each $\bb \in \nu|_{r-1}(W)$. \label{sigma:maximal}

\medskip
\item For each $\bc \in \nu|_{r-1}(V) + \nu|_{r-1}(W)$, there is a divisor in $V(\ba)\cdot W(\bb)$ which is maximal in $(V\cdot W)(\bc)$, for some $\ba+\bb=\bc$. \label{sigma:additive}

\medskip
\item For each $\ba \in \nu|_{r-1}(V)$ and $\bb \in \nu|_{r-1}(W)$, the linear systems $V(\ba)$ and $W(\bb)$ are saturated with respect to $Y_r$. \label{sigma:saturated}
\end{enumerate}
We say \define{condition $\cC_r(V,W)$ holds} if \eqref{sigma:maximal} and \eqref{sigma:additive} hold.  We say the \define{condition $\cD_r(V,W)$ holds} if \eqref{sigma:maximal}, \eqref{sigma:additive}, \eqref{sigma:saturated}, and $\cD_{r-1}(V,W)$ hold. 
\end{definition}

\begin{proposition}\label{p:plus}
Suppose $\cC_r(V,W)$ holds, and $\nu|_{r-1}(V\cdot W) = \nu|_{r-1}(V)+\nu|_{r-1}(W)$.  Then $\Conv(\nu|_r(V\cdot W)) = \Conv(\nu|_r(V)) + \Conv(\nu|_r(W))$.

Suppose, additionally, that the saturation condition \eqref{sigma:saturated} holds.  
Then $\nu|_r(V\cdot W) = \nu|_r(V) + \nu|_r(W)$.
\end{proposition}

\begin{proof}
Given $(c_1,\ldots,c_r) \in \nu|_r(V\cdot W)$, write $\bc=(c_1,\ldots,c_{r-1})$.  Find $\ba\in\nu|_{r-1}(V)$ and $\bb\in\nu|_{r-1}(W)$ such that $\ba+\bb=\bc$, and such that $V(\ba)\cdot W(\bb)$ contains a divisor which is maximal in $(V\cdot W)(\bc)$.  Such a divisor is also maximal in $V(\ba)\cdot W(\bb)$, and therefore its order of vanishing along $Y_r$ is equal to $a_r^\mx + b_r^\mx$, where $a_r^\mx$ and $b_r^\mx$ are the orders of maximal divisors in $V(\ba)$ and $W(\bb)$, respectively.  It follows that $c_r \leq c_r^\mx := a_r^\mx + b_r^\mx$, proving the first statement.

If $V(\ba)$ and $W(\bb)$ are saturated, then for each $c_r$ in the interval $[0,c_r^\mx]$, there are sections $s \in V(\ba)$ and $t \in W(\bb)$ with $\ord_{Y_r}(s)=a_r$, $\ord_{Y_r}(t)=b_r$, and $a_r+b_r=c_r$.  These correspond to points $(\ba,a_r) \in \nu|_r(V)$ and $(\bb,b_r) \in \nu|_r(W)$ with $(\ba,a_r)+(\bb,b_r)=(\bc,c_r)$, so the second statement follows.
\end{proof}

From the inductive nature of condition $\cD_r(V,W)$, we have the following:

\begin{corollary}\label{c:cond-D}
Suppose $\cD_r(V,V^m)$ holds for all $m\geq 1$.  Then $\Gamma_\nu|_r(V)$ is generated by $\{1\}\times\nu|_r(V)$.  In particular, if $r=d$, then the Okounkov body $\Delta_{Y_\bullet}(V)$ is equal to the lattice polytope $\Conv(\nu(V))$. \qed
\end{corollary}

\section{Degenerations}\label{s:deg}

Return to the general setup of \S\ref{s:vals}.  We have a filtration of $R=R(V)$ indexed by $\Gamma=\Gamma_\nu(V)$, since $R_{\leq (m,u)}\subseteq R_{\leq (m',u')}$ whenever $(m,u)\leq (m',u')$.  Let
\[
  \gr R = \bigoplus_{(m,u)\in \Gamma} R_{\leq (m,u)}/R_{<(m,u)}
\]
be the associated graded; this is a $\Gamma$-graded ring.  Our first observation is that the degenerations constructed in \cite{caldero} and \cite{ab} generalize to this setting.

\begin{proposition}\label{p:ring-deg}
Let $R = R(V)$, and assume that $\gr R$ is finitely generated (so $\Gamma=\Gamma_\nu(V)$ is also finitely generated).  Then there is a finitely generated, $\N$-graded, flat $\kk[t]$-subalegbra $\RR\subseteq R[t]$, such that
\begin{itemize}
\item $\RR/t\RR \isom \gr R$ and 
\item $\RR[t^{-1}] \isom R[t,t^{-1}]$ as $\kk[t,t^{-1}]$-algebras.
\end{itemize}

More specifically, there is a linear projection $\pi:\Z\times\Z^d \to \Z$, inducing an $\N$-filtration $R_{\leq k} \subseteq R$ whose associated graded is $\gr R$.  The Rees algebra
\[
  \RR = \bigoplus_{k\geq 0}\left( R_{\leq k}\right) t^k
\]
for this filtration satisfies the required properties.

Finally, the $\N$-grading on $R = \bigoplus V^m$ is compatible with the one 
on $\RR$ (via powers of $t$), so in fact $\RR$ is naturally $(\N\times\N)$-graded.
\end{proposition}

\noindent
The proposition is proved by imitating the proof of \cite[Proposition 2.2]{ab}, which in turn is based on \cite[\S3.2]{caldero}.  We repeat the arguments here, since we will adapt some of them for Proposition~\ref{p:compatible}.

First, we need a lemma.
\begin{lemma}\label{l:caldero}
Let $\mathcal{S}$ be a finite subset of $\Z\times\Z^d$.  Then there is a linear projection $\pi\colon\Z\times\Z^d \to \Z$ such that $\pi(m,u)<\pi(n,v)$ whenever $(m,u)<(n,v)$ in $\mathcal{S}$ (using the modified lexicographic order on $\Z\times\Z^d$, from \S\ref{s:vals}).
\end{lemma}

\noindent
This is a standard fact, and can be found in \cite[Proposition~1.8]{bayer} or \cite[Lemma~3.2]{caldero}.  Its proof is simple:

\begin{proof}
Let $C$ be a positive integer larger than all coordinates of $(m,u)-(n,v)$, for every pair 
of elements $(m,u),(n,v)\in\mathcal{S}$, and let $\alpha_0,\ldots,\alpha_d$ be chosen 
so that $\alpha_k > C(\alpha_{k+1} + \cdots + \alpha_d)$ for each $k\geq 0$.  Then 
$\pi = \alpha_0 e_0^* - \sum_{i=1}^d \alpha_i e_i^*$ does the trick, where $e_i^*:\Z\times\Z^d \to \Z$ is the $i$th coordinate.
\end{proof}

We now prove the proposition.  Simultaneously, we will give a presentation for the ring $\RR$.

\begin{proof}[Proof of Proposition~\ref{p:ring-deg}]
Choose homogeneous generators $\bar{f}_1,\ldots,\bar{f}_p$ for $\gr R$, with 
$\deg(\bar{f}_i) = (m_i,u_i) \in \Gamma \subseteq \N\times \Z^d$.  Lift these 
to generators $f_1,\ldots,f_p$ for $R$, with $f_i\in V^{m_i}$.  Set $S=\kk[x_1,\ldots,x_p]$, and give $S$ a grading by $\deg(x_i) = (m_i,u_i)$; thus the surjective map $S \to \gr R$ defined by $x_i \mapsto \bar{f}_i$ is a map of graded rings.  Let $\bar{g}_1,\ldots,\bar{g}_q \in S$ be homogeneous generators of the kernel, and set $\deg(\bar{g}_j) = (n_j,v_j)$.  It follows that $\bar{g}_j(f_1,\ldots,f_p)$ lies in $R_{<(n_j,v_j)}$ for each $j$.  Since this space is finite-dimensional and $\Gamma_{<(n_j,v_j)}$ is finite, one can find elements $g_j$ in $\bar{g}_j + S_{<(n_j,v_j)}$ such that $g_j(f_1,\ldots,f_p)=0$.  The $g_j$ will not be homogeneous for the full $\N\times\Z^d$ grading of $S$, but since the $f_i$ are homogeneous for the first ($\N$) factor, the $g_j$ can be chosen to respect the $\N$-grading as well.  (This means $g_j$ lies in $(S_{n_j})_{>v_j}$, where $S_{m}$ denotes the piece of $S$ with $\N$-degree $m$.) 

We now claim that the induced map $S/(g_1,\ldots,g_q) \to R$ is an isomorphism.  To see this, let $I$ denote the kernel of the map $S \to R$, and let $J$ be the kernel of $S \to \gr R$.  Then $g_j \in I$ by construction, and the initial terms $\bar{g}_j$ generate $J$; we only need to show that the $g_j$'s generate $I$.  In fact, $J$ is the initial ideal of $I$, using the term order determined by the order on $\Z\times\Z^d$.  (Specifically, the term order is $x_1^{a_1}\cdots x_p^{a_p} \leq x_1^{b_1} \cdots x_p^{b_p}$ iff $\sum a_i(m_i,u_i) \leq \sum b_i(m_i,u_i)$.  Note that we allow ties between monomials in this notion of term order.)  Indeed, the inclusion $J\subseteq \mathrm{in}(I)$ is clear; and given an element $h\in I$, the equality $h(f_1,\ldots,f_p)=0$ implies $\bar{h}(\bar{f}_1,\ldots,\bar{f}_p)=0$, so the initial term $\bar{h}$ lies in $J$.  It follows that the $g_j$'s form a Gr\"obner basis for $I$, using, e.g., \cite[Exercise~15.14(a)]{eisenbud}, and in particular they generate the ideal.  

It is a standard fact that Gr\"obner bases give rise to flat degenerations, but for later use, it will be convenient to recall the construction.  
%
To do this, let $(n'_j,v'_j)$ be a degree of some homogeneous component of 
$g_j-\bar{g}_j$, so $(n'_j,v'_j)<(n_j,v_j)$; consider the difference $(n_j,v_j)-(n'_j,v'_j)$.  (The choice of $g_j$ actually means $n'_j=n_j$, but we will preserve the notation for clarity.)  Let $\mathcal{S} \subseteq \Z \times \Z^d$ be the finite set consisting of such differences, together with $0$ and the $d+1$ generators of $\N \times \N^d$.  Using Lemma~\ref{l:caldero}, choose $\pi\colon\Z\times\Z^d \to \Z$ preserving order on $\mathcal{S}$.  Since each nonzero element of $\mathcal{S}$ is greater than $0$, $\pi$ takes positive integer values on $\mathcal{S}\setminus\{0\}$, and therefore also takes positive values on the cone it spans; in particular, $\pi$ takes positive values on $\N\times\N^d$.  Therefore, for each $j$, we have $\pi(n_j,v_j)>\pi(n'_j,v'_j)>0$---that is, the initial term of $g_j$ with respect to the weighting of monomials defined by $\pi$ is exactly $\bar{g}_j$, the initial term with respect to the original order.

Now let $w_i=\pi(m_i,u_i)$ be the degree of $x_i$ under the weighting induced by $\pi$, let $k_j=\pi(n_j,v_j)$, and set $\tilde{g}_j = \tau^{k_j} g_j( \tau^{-w_1}x_1,\ldots,\tau^{-w_p}x_p )$.  Set 
\[
  \RR = S[\tau]/(\tilde{g}_1,\ldots,\tilde{g}_q).
\]
From the construction, this ring satisfies the properties specified in the proposition, where the $\kk[t]$-algebra structure is given by $t\mapsto \tau$.

The ring $\RR$ has two $\N$-gradings, which we will call the ``$m$-grading'' and $t$-grading''.  The first comes from the $\N$-grading on $S$, which $\RR$ inherits since the $\tilde{g}_j$'s are homogeneous for this grading, of degree $n_j$.  The second, which we will call the ``$t$-grading,'' is defined by giving $x_i$ degree $w_i$, and $\tau$ degree $1$.  Then $\tau^{-w_i}x_i$ has $t$-degree $0$, so each $\tilde{g}_j$ is homogeneous for the $t$-grading, of degree $k_j$.  


To see how $\RR$ arises as a Rees ring, as in \cite[\S3.2]{caldero} and \cite[proof of Prop. 2.2]{ab}, one can define an $\N$-filtration by
\[
  R_{\leq k} = \Span\{ F = f_{i_1}^{a_1} \cdots f_{i_s}^{a_s} \,|\, \pi\circ\hat\nu(F) \leq k \},
\]
and this has the same associated graded $\gr R$ as the one from the $\Gamma$-filtration.  The corresponding Rees ring is isomorphic to $\RR$; in fact, the presentation given above is a standard way of producing a Rees ring from a Gr\"obner basis (see, e.g., \cite[Theorem~15.17]{eisenbud} or \cite[\S7.A]{bg}).  Explicitly, the map $S[\tau] \to \bigoplus (R_{\leq k}) t^k$ is given by $x_i \mapsto t^{w_i}f_i$ and $\tau \mapsto t$.
\end{proof}

\begin{example}\label{ex:cubic-algebra}
Suppose $R=R(V)$ is the ring $\kk[x,y,z]/(g)$, where $g=x^3+x^2z+xz^2+y^3$, and the valuation is $\nu(x)=3$, $\nu(y)=1$, and $\nu(z)=0$.  The ring is generated in degree $1$, so the full $\N\times\Z$ degrees of the variables are $\deg(x)=(1,3)$, $\deg(y)=(1,1)$, $\deg(z)=(1,0)$.  We can take $\pi\colon \Z\times\Z \to \Z$ to be given by $\pi(a,b)=5a-b$, so $w_x=\pi(1,3)=2$, $w_y=\pi(1,1)=4$, and $w_z=\pi(1,0)=5$.  The initial term of $g$ is $\bar{g}=xz^2+y^3$, which has weight $\pi(3,3)=12$.  So
\begin{align*}
  \tilde{g} &= \tau^{12}(\tau^{-6}x^3+\tau^{-9}x^2z + \tau^{-12}xz^2 + \tau^{-12}y^3) \\
            &= \tau^6 x^3 + \tau^3x^2z + xz^2 + y^3.
\end{align*}
The ring $\RR$ is $\kk[x,y,z,\tau]/(\tilde{g})$, giving the flat degeneration of $R$ to $\RR/t\RR = \kk[x,y,z]/(xz^2+y^3)$.
\end{example}

Geometrically, Proposition~\ref{p:ring-deg} says there is a flat family of 
affine varieties $\hat\XX \to \A^1$, such that the general fiber $\hat{X}_t$ is 
isomorphic to $\hat{X} = \Spec R(V)$ for $t\neq 0$, and the zero fiber $\hat{X}_0$ 
has an action of the torus $\kk^*\times (\kk^*)^d$.

The first $\N$-grading, coming from the grading of $\kk[t]$, says that the 
family is equivariant for actions of $\kk^*$ on $\hat\XX$ and $\A^1$.

Taking $\Proj$ with respect to the second $\N$-grading, coming from 
the grading of $R$, we obtain a projective flat family $\XX=\Proj\RR \to 
\A^1$, with general fiber isomorphic to $X=X(V)=\Proj R(V)$ and special fiber 
$X_0 = \Proj (\gr R)$ equipped with an action of $(\kk^*)^d$.  The $\kk^*$ action from 
the first $\N$-grading descends, and $\XX \to \A^1$ is equivariant for this $\kk^*$ action.

As in \cite[Theorem~3.2]{ab}, the polarization of $X$ also deforms.  The $R$-module $R(1)$ obtained by shifting the $\N$-grading corresponds to the very ample line bundle $\O_X(1)$ embedding $X$ in projective space, since by definition $R$ is generated by its first graded piece.  Similarly, shifting the grading on $\gr R$ produces a sheaf $\O_{X_0}(1)$, but this may not be locally free, since $\Gamma$ (and $\gr R$) may not be generated in degree one.

In fact, letting $(m_1,u_1),\ldots,(m_p,u_p)$ be the valuations of the generators $f_1,\ldots,f_p$ chosen in the proof of Proposition~\ref{p:ring-deg}, we get an embedding
\[
  \XX \hookrightarrow \P(m_1,\ldots,m_p) \times \A^1
\]
into weighted projective space over $\A^1$.  Pulling back the Serre sheaf for weighted projective space, we obtain a sheaf $\O_\XX(1)$ on $\XX$ which restricts to $\O_X(1)$ on the fibers $X_t=X$ for $t\neq 0$, and to $\O_{X_0}(1)$ at $t=0$.  There are analogous sheaves $\O_\XX(n)$ for all $n\geq 0$, corresponding to higher shifts of the grading.  (Note, however, that $\O_\XX(n)$ is not equal to $\O_\XX(1)^{\otimes n}$ in general.)  For $n=\lcm(m_1,\ldots,m_p)$, the sheaf $\O_\XX(n)$ is a very ample line bundle, coming from the restriction of the generator of $\Pic(\P(m_1,\ldots,m_p))$.

We can summarize the above geometric interpretation of Proposition~\ref{p:ring-deg} as follows.  Let $X$ be a projective variety with a very ample and normally generated line bundle $\cL$, with $V=H^0(X,\cL)$, so the section ring of $\cL$ is equal to $R(V)$.

\begin{corollary}\label{c:geom-deg}
Given such a pair $(X,\cL)$, assume the hypothesis of Proposition~\ref{p:ring-deg}.  There is a flat, projective morphism $\XX \to \A^1$, equivariant for an action of $\kk^*$, such that the zero fiber $X_0$ is equipped with an action of $(\kk^*)^d$, and its complement $\XX\setminus X_0$ is isomorphic to $X \times (\A^1\setminus\{0\})$.  Moreover, there are sheaves $\O_\XX(n)$ on $\XX$ restricting to the line bundles $\cL^{\otimes n}$ on fibers $X_t=X$ for $t\neq 0$, and restricting to sheaves $\cL_0^{(n)}$ on $X_0$.  For sufficiently divisible $n$, these sheaves are very ample line bundles.  

If, additionally, $X$ is normal and $X_0$ is reduced, then $\XX$ is normal, and the sheaves $\O_\XX(n)$ are reflexive (and divisorial).
\end{corollary}

\begin{proof}
The statements in the first paragraph are direct translations of the content of Proposition~\ref{p:ring-deg}, as explained above.  Note that, in general, the singular locus of $\XX$ is contained in the union of the singular loci of $X_t$, for all $t$.  For $t\neq 0$, this is obvious, since $\XX$ is trivial away from $0$.  On the other hand, let $A$ be a local ring of $\XX$ at a point in the zero fiber which is nonsingular in $X_0$.  Since $\XX = \Proj \RR$ is flat over $\A^1 = \Spec \kk[t]$, $t$ is a nonzerodivisor in $A$.  Since the point is nonsingular in $X_0$, the maximal ideal of the local ring $A/(t)$ is generated by a regular sequence $(a_1,\ldots,a_{\dim X_0})$.  Then $(t,a_1,\ldots,a_{\dim X_0})$ is a regular sequence generating the maximal ideal of $A$.

It follows that $\XX$ is always $\mathrm{R}_1$.  When $X$ is normal, the non-normal locus of $\XX$ must be contained in the zero fiber.  But if $X_0$ is reduced, the local rings $A/(t)$ are $\mathrm{S}_1$, so the local rings $A$ are $\mathrm{S}_2$, and $\XX$ is normal by Serre's criterion (cf.~\cite[Exercise~2.2.33]{bh}).  The claim about reflexivity of $\O_\XX(n)$ is proved as in \cite[Theorem~3.2]{ab}.
\end{proof}

\begin{remark}
The limit $X_0$ depends only on the valuation $\nu$, but the family itself 
depends on the choice of projection $\pi:\Z\times\Z^d\to\Z$.
\end{remark}

\begin{remark}
A version of Proposition~\ref{p:ring-deg} is stated in \cite[\S5.6]{kk-alg}, where the language of SAGBI bases is used.  Degenerations to graded rings using more general valuations were constructed by Teissier \cite[\S2]{teissier}, with a view toward resolution of singularities; the algebras he considers are not necessarily finitely generated.
\end{remark}

\begin{remark}\label{rmk:multigraded}
Given $r$ linear systems $V_1,\ldots,V_r$, with $V_i\subseteq H^0(X,\cL_i)$, one can form the multigraded section ring $R(V_1,\ldots,V_r) = \bigoplus_{\bm \in \N^r} \bV^\bm$, where $\bV^\bm = V_1^{m_1}\cdot V_2^{m_2} \cdots V_r^{m_r} \subseteq H^0(X,\cL_1^{\otimes m_1} \otimes \cdots \otimes \cL^{\otimes m_r})$. The $\N^r$-graded structure of $R$ gives rise to a multi-parameter ``degeneration in stages'' generalizing the Gelfand-Tsetlin degeneration of \cite{km}.  The extended valuation $\hat\nu$ is defined as in \S\ref{s:vals}, but now takes values in $\N^r\times\Z^d$.  One begins by considering $R$ itself as an $(\N^r \times \Z^d)$-graded (in fact, $\Gamma_\nu$-graded) ring, but with the $(\bm,u)$-component equal to zero for $u\neq 0$.  Then one constructs a sequence of filtrations and associated graded rings
\[
  R = \gr^{(0)}R,\; \gr^{(1)}R,\; \ldots, \; \gr^{(r)}R = \gr R.
\]
At the $j$th stage, the subring of $\gr^{(j)}R$ indexed by $\N^j \times\{0\}$ has a nontrivial $\Gamma$-grading, while the subring indexed by $\{0\}\times \N^{r-j}$ remains isomorphic to $R(V_{j+1},\ldots,V_r)$.  The details of the construction are straightforward but notationally quite unwieldy, so we omit them.
\end{remark}

When the valuation comes from a complete flag, we obtain a toric 
degeneration.  Let $X$ be a projective variety, and let $V\subseteq H^0(X,\cL)$ be a linear system.  Let $X(V) = \Proj R(V)$ be the (closure of the) image of $X$ under the corresponding rational map to $\P(V)$.

\begin{theorem}\label{t:deg}
Fix a complete flag $Y_\bullet$ on $X$, and let $\nu=\nu_{Y_\bullet}$ be the associated valuation.  Assume $\Gamma=\Gamma_\nu(V)$ is finitely generated.  Then $X(V)$ admits a flat degeneration to the (not necessarily normal) toric variety $X(\Gamma) = \Proj \kk[\Gamma]$.

Suppose $\Cone(\Gamma)$ is generated by $\Gamma\cap(\{1\}\times\N^d)$.  If $\nu(V^d) = d\Delta_{Y_\bullet}(V) \cap \Z^d$, then $X(\Gamma)$ is normal; when $V=H^0(X,\cL)$ with $\cL$ very ample and normally generated, the converse holds.
\end{theorem}

\begin{proof}
It suffices to show that the associated graded $\gr R$ is isomorphic to 
the semigroup algebra $\kk[\Gamma]$.  By Lemma~\ref{l:dim1}, $R_{\leq (m,u)}/R_{<(m,u)} = (V^m)_{\geq u}/(V^m)_{>u}$ is one-dimensional when $(m,u)\in \Gamma$, and it is zero otherwise.  Moreover, homogeneous elements of $\gr R$ are nonzerodivisors, since $R_{\leq (m,u)}\cdot R_{\leq (m',u')} \not\subseteq R_{<(m+m',u+u')}$, as one sees from the additivity of the valuation $\nu$.  It follows that $\gr R \isom \kk[\Gamma]$ (see \cite[Remark~4.13]{bg}).

The second statement is an immediate consequence of Lemma~\ref{l:normal}.  The only possible exception to the ``converse'' statement happens when $\Gamma$ generates a proper sublattice of $\Z\times \Z^d$; however, by \cite[Lemma~2.2]{lm}, this cannot occur for complete linear series.
\end{proof}

Note that the condition $\cD_d(V,V^m)$ for all $m\geq 0$ implies all the hypotheses of the theorem.

\begin{example}\label{ex:curve-good}
Let $X$ be an elliptic curve, with flag $Y_\bullet$ given by $X\ni P$ for some point $P$.  Take $V=H^0(X,\O_X(3P))$, giving the cubic embedding in $\P^2$.  Then $\cD_d(V,V^m)$ holds for all $m\geq 0$ (with respect to $Y_\bullet$), and one sees that the semigroup $\Gamma\subset \N\times\Z$ is generated by $(1,0)$, $(1,1)$, and $(1,3)$.  (See Figure~\ref{f:elliptic}(a).)  The curve 
$X$ degenerates to a cuspidal cubic curve (with its toric structure).  For the curve $X = \{ x^3+x^2z+xz^2+y^3=0 \} \subseteq \P^2$, with inflection point $P=[0,0,1]$, the algebra of the degeneration is given in Example~\ref{ex:cubic-algebra}.
\end{example}

\begin{example}[{\cite[Ex. 1.7]{lm}}]\label{ex:curve-bad}
Take $X$ to be an elliptic curve and $V=H^0(X,\O(3P))$ as before, but using a general point $P'\in X$ to define the flag.  The semigroup is $\Gamma_{P'}(V) = \{(0,0)\} \cup \{ (m,r) \,|\, 0\leq r\leq 3m-1\} \subset \N\times\Z$, which is not finitely generated: every lattice point on the line $r=3m-1$ is needed to generate $\Gamma$.  (Figure~\ref{f:elliptic}(b).)
\end{example}

\begin{remark}
Examples~\ref{ex:curve-good} and \ref{ex:curve-bad} work just as well for any smooth projective curve, using a very ample linear system $V$ (\cite[Ex. 1.7]{lm}).
\end{remark}

\begin{figure}
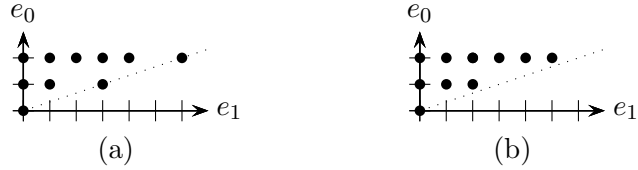

\pspicture(-70,-20)(300,50)
\psaxes[Dx=10,Dy=10,labels=none]{->}(70,30)
\rput(0,37){$e_0$}
\rput(78,0){$e_1$}
\pscircle*(0,0){2}
\pscircle*(0,10){2}
\pscircle*(10,10){2}
\pscircle*(30,10){2}
\pscircle*(0,20){2}
\pscircle*(10,20){2}
\pscircle*(20,20){2}
\pscircle*(30,20){2}
\pscircle*(40,20){2}
\pscircle*(60,20){2}

\psline[linestyle=dotted]{-}(0,0)(69,23)

\rput(35,-15){(a)}

\psaxes[Dx=10,Dy=10,labels=none]{->}(150,0)(150,0)(220,30)
\rput(150,37){$e_0$}
\rput(228,0){$e_1$}
\pscircle*(150,0){2}
\pscircle*(150,10){2}
\pscircle*(160,10){2}
\pscircle*(170,10){2}
\pscircle*(170,20){2}
\pscircle*(150,20){2}
\pscircle*(160,20){2}
\pscircle*(170,20){2}
\pscircle*(180,20){2}
\pscircle*(190,20){2}
\pscircle*(200,20){2}

\psline[linestyle=dotted]{-}(150,0)(219,23)

\rput(185,-15){(b)}

\endpspicture

\caption{\small Semigroups $\Gamma_\nu$ from Examples~\ref{ex:curve-good} and \ref{ex:curve-bad}.  (The standard basis for $\Z\times \Z$ is $\{e_0,e_1\}$; the axes have been interchanged for typographical convenience.)}
\label{f:elliptic}
\end{figure}

The construction of Proposition~\ref{p:ring-deg} is functorial, in a sense to be made precise by Proposition~\ref{p:compatible}.  For this, we need to define two notions of compatibility.  
First, let $M$ be a finitely generated free abelian group.  Suppose the vector space $V$ is $M$-graded, i.e., we are given a decomposition $V=\bigoplus_{\lambda\in M} V_\lambda$, and that the kernel of the natural map $\Sym^\bullet V \to R(V)$ is homogeneous, so the grading on $V$ defines one on $R$.  (This $M$-grading is independent of the standard $\N$-grading on $R$.)

\begin{definition}\label{d:m-compat}
The $M$-grading is \define{compatible} with a $\Z^d$-valuation 
$\nu$ if for all $m\in\N$ and $u\in\Z^d$, the subspaces $V^m_{\geq u}$ and $V^m_{>u}$ are $M$-graded subspaces (i.e., they have bases of $M$-homogeneous elements).
\end{definition}

\begin{remark}
In fact, it suffices to require that $V^m_{\geq u}$ and $V^m_{>u}$ are $M$-graded as $(m,u)$ ranges over a generating set for $\Gamma_\nu\subseteq\N\times\Z^d$.
\end{remark}

Second, let $R$ and $R'$ be $\N$-graded rings, and consider a graded ring homomorphism $\phi\colon R\to R'$, with kernel $J\subseteq R$.  Let $h\colon\Z^d \to \Z^{d'}$ be a map of ordered groups, with respect to the \emph{opposite} lexicographic order; this makes $\id\times h\colon \Z\times \Z^d \to\Z\times \Z^{d'}$ into a map of ordered groups with respect to the order defined in \S\ref{s:vals}.  Let $\nu$ be a $\Z^d$-valuation on $R$, and $\nu'$ be a $\Z^{d'}$-valuation on $R'$.

\begin{definition}\label{d:compat2}
The valuations $\nu$ and $\nu'$ are \define{compatible} with $\phi$ and $h$ if the diagram
\begin{diagram}
R\setminus J     &  \rTo^\phi   &  R'\setminus\{0\} \\
\dTo^\nu &        & \dTo_{\nu'} \\
\Z^d  & \rTo^h      & \Z^{d'}
\end{diagram}
commutes.
\end{definition}

\begin{remark}
The requirement that $h$ preserve the opposite lexicographic order is somewhat restrictive.  For example, if $d\geq d'$ and $h$ is a coordinate projection, then it must be the projection 
on the last $d'$ factors.
\end{remark}

\begin{proposition}\label{p:compatible}
Let $R=R(V)$ and $\nu$ be as in Proposition~\ref{p:ring-deg}.
\renewcommand{\theenumi}{\alph{enumi}}
\begin{enumerate}
\item If $R$ has an $M$-grading which is compatible with $\nu$, then the grading lifts to $\RR$, so $\RR$ is $(\N\times M\times\N)$-graded. \label{p:compatible-grading}

\smallskip

\item Let $V'$, $R'=R(V')$, and $\nu'$ be as in Proposition~\ref{p:ring-deg}, 
as well.  Let $h\colon\Z^d \to \Z^{d'}$ be a map of ordered groups, and let $\phi\colon R \to R'$ be a graded ring map arising from a linear map $V \to V'$.  Suppose these are compatible with $\nu$ and $\nu'$.  Then one can choose 
$\pi'\colon\Z\times\Z^{d'}\to \Z$, and set $\pi = \pi'\circ (\id\times h)$, such that the corresponding Rees rings $\RR$ and $\RR'$ both satisfy the properties specified in Proposition~\ref{p:ring-deg}.  Moreover, there is an induced map of $\kk[t]$-algebras $\Phi\colon\RR \to \RR'$, which preserves the $(\N\times\N)$-gradings.  \label{p:compatible-map}
\end{enumerate}
\end{proposition}

\begin{proof}
For \eqref{p:compatible-grading}, simply observe that by definition, $R_{\leq (m,u)}$ is $M$-graded for each $(m,u)\in\Z\times\Z^d$; therefore so is $R_{\leq k}$, once a projection $\pi$ is chosen.

For \eqref{p:compatible-map}, apply Lemma~\ref{l:caldero} to the finite set $\mathcal{S}$ consisting of the elements used in the construction of $\RR'$, together with the images of those used in the construction of $\RR$ under the map $h$.  Since $\pi=\pi'\circ (\id\times h)$, and $\nu$ and $\nu'$ are compatible, we have $\phi(R_{\leq k}) \subseteq R'_{\leq k}$ for each $k$.  The statement follows.
\end{proof}

These notions of compatibility arise naturally when $\nu = \nu_{Y_\bullet}$.

\begin{example}\label{ex:torus}
Let $T$ be a torus with character group $M$ acting on $X$, let $\cL$ be a $T$-linearized line bundle, with the induced $M$-grading on $H^0(X,\cL)$, and let $V\subseteq H^0(X,\cL)$ be an $M$-graded subspace.  Let $Y_\bullet$ be a flag 
of $T$-invariant subvarieties.  Then the $M$-grading is compatible with $\nu_{Y_{\bullet}}$.  

To see this, first note that the maps $H^0(X,\cL)^{\otimes m} \to H^0(X,\cL^{\otimes m})$ are $T$-equivariant for all $m$, and it follows that the map $\Sym^\bullet V \to R(V)$ defines an $M$-grading on $R=R(V)$.  Next, observe that $\nu(t\cdot s) = \nu(s)$ for any section $s\in H^0(X,\cL)$ and any $t\in T$.  (This follows from the general fact that $\ord_{D}(g\cdot s) = \ord_{g^{-1}D}(s)$ for any divisor $D$ and any group action.)  Consequently,  for any $T$-invariant subspace $V\subseteq H^0(X,\cL)$ and any $u\in \Z^d$, the subspace $V_{\geq u}$ is also $T$-invariant.  Replacing $V$ with $V^m$, we see that the $M$-grading is compatible.

Now let $(y_1,\ldots,y_d)$ be a regular sequence at the point $Y_d$, with 
$(y_1,\ldots,y_r)$ defining $Y_r$ for all $1\leq r\leq d$.  Since the $y_i$'s 
define a basis for the cotangent space $T^*_{Y_d}X$, the torus $T$ acts on $T_{Y_d}X$ with weights $-\lambda_1,\ldots,-\lambda_d$.  In fact, the map $\Z^d \to M$, $e_i\mapsto\lambda_i$, gives a map $T \to (\kk^*)^d$ and defines an action of $T$ on the limit toric variety $X(\Delta_{Y_\bullet}(V))$.
\end{example}

\begin{example}\label{ex:birational}
Suppose $V' \hookrightarrow V$ is an inclusion which induces a birational map 
$X(V) \to X(V')$.  Fix any flag $Y_\bullet$ on $X$, and let $h:\Z^d\to \Z^d$ be the identity.  The corresponding ring map $R(V')\to R(V)$ is compatible with $\nu$.

More specifically, let $\cL$ be a very ample line bundle on $X$, set $V=H^0(X,\cL)$, and assume $\Gamma_{Y_\bullet}(V)$ is finitely generated.  Let $V'\subseteq V$ be a subspace such that the corresponding rational map $\phi:X \to X'=X(V')$ is a birational morphism, and assume that $\Gamma_{Y_\bullet}(V')$ is also finitely generated.  Then $\Delta_{Y_\bullet}(V')\subseteq \Delta_{Y_\bullet}(V)$, and the birational morphism $\phi\colon X \to X'$ degenerates to a birational morphism of toric varieties $\phi_0\colon X(\Delta_{Y_\bullet}(V)) \to X(\Delta_{Y_\bullet}(V'))$.  That is, there is a diagram
\begin{diagram}
\XX & & \rTo^{\Phi} & & \XX' \\
        & \rdTo &    & \ldTo \\
        &      &  \A^1,
\end{diagram}
such that $\Phi_t \isom \phi$ for $t\neq 0$, and $\Phi_0 = \phi_0$.
\end{example}

\begin{example}\label{ex:subvar}
Let $R=R(V)$, $X=X(V)$, and $\nu=\nu_{Y_\bullet}$.  Let $X'=Y_r\subseteq X$, so $X'$ has dimension $d-r$, and define the restricted flag $Y'_\bullet$ to be $X'=Y_r \supset Y_{r+1} \supset \cdots \supset Y_d$.  Assume that $X'$ is not contained in the base locus of $V$.  Let $J\subseteq R$ be the (largest) homogeneous ideal of $X'$, and let $\phi:R \to R'=R/J$ be the natural map.  Then $\nu_{Y_\bullet}$ and $\nu_{Y'_\bullet}$ are compatible with $\phi$ and $h$, where $h:\Z^d\to\Z^{d-r}$ is the projection on the last $d-r$ coordinates.
\end{example}

\section{Examples}\label{s:examples}

\subsection{Ruled surfaces}

We use the conventions of \cite[\S5.2]{ha}.  Let $\pi:X=\P(\mathcal{E})\to C$ be ruled surface over a curve of genus $g$, normalized so that $H^0(C,\mathcal{E}\otimes\cL) = 0$ for any line bundle $\cL$ of negative degree on $C$.  Let $e=-\deg(\mathcal{E})$; we will assume $e\geq 0$ here.  Let $C_0$ be a section of $\pi$ corresponding to $\O_{\P(\mathcal{E})}(1)$.  Fix a point $p\in C$, and let $F = \pi^{-1}(p)$ be the fiber.  A divisor $D=aF + bC_0$ is effective iff $a,b\geq 0$, and nef iff $0\leq b\leq \frac{a}{e}$.  

Let $Y_\bullet$ be the flag $X \supset F \supset \{x\}$, where $x=C_0\cap F$.  Fix a big divisor $D=aF + bC_0$; this is maximal in its complete linear system with respect to $F$.  Replacing $D$ with a multiple if necessary, assume that $be\geq 2g$, and $\O(D)$ is normally and globally generated.  Set $V=H^0(X,\O(D))$, so $V^m = H^0(X,\O(mD))$.

\begin{proposition}
Suppose $a>be>0$, so $D$ is ample.  Then $\Gamma_\nu(V)$ is finitely generated, and $X$ degenerates to a toric variety whose normalization is isomorphic to the Hirzebruch surface $\mathbf{F}_{e}$.
\end{proposition}

\begin{proof}
Using \cite[\S6.2]{lm}, we find that the Okounkov body $\Delta_{Y_\bullet}(V)$ is the trapezoid with vertices at $(0,0)$, $(0,b)$, $(a,b)$, and $(a-be,0)$.  It is well-known that the toric variety corresponding to this polytope is $\mathbf{F}_e$ (see, e.g., \cite[\S1.1]{fulton}).  We will check that these four points occur as valuations of sections in $V$; the claim then follows from Proposition~\ref{p:fg-easy}.

First, we have assumed $\O(D)$ is globally generated, so there is a section not vanishing at $x$; this accounts for $(0,0)$.  Since $a>be\geq 2g$, the divisor $\O_C(a\cdot p)$ is globally generated, and it follows that $D$ is linearly equivalent to $E+bC_0$, for some effective divisor $E$ pulled back from $C$ and not containing $F$; the corresponding section has valuation $(0,b)$.  The section defining divisor $D$ itself evaluates to $(a,b)$.

To find a section with valuation $(a-be,0)$, note that $be\geq 2g$ implies there is a point $q\neq p$ with $\O_C(be\cdot p)\isom\O_C(be\cdot q)$.  It follows that $D$ is linearly equivalent to $(a-be)F+beF' + bC_0$, so there are sections $s$ with $\nu_1(s) = a-be$.  Now $C_0\cdot(D-(a-be)F) = C_0\cdot (beF + bC_0) = 0$, so there is a section of $\O(D-(a-be)F)$ not vanishing at $x$.  The image of this section under the map $H^0(X, \O(beF'+bC_0)) \to H^0(X,\O(D))$ is the desired section.
\end{proof}

A similar argument works when $a\leq be$, in which case $X(V)$ degenerates to a toric variety whose normalization is the cone over the rational normal curve of degree $e$.

\subsection{Abelian surfaces}

Let $X$ be an abelian surface with Picard number at least $3$.  Let $C$ be a curve such that the line bundle $\O(C)$ is normally generated, and let $C_0$ be a curve with $(C_0^2)=0$, not numerically equivalent to a multiple of $C$.  Write $n=(C^2)$ and $m=(C\cdot C_0)$, so $m>0$, and it follows that $D=C+C_0$ is ample.  For the flag $Y_\bullet$, take $X \supset C \supset \{x\}$, where $x \in C \cap C_0$.

\begin{proposition}\label{p:abelian}
Assume that $\O(C)|_C \isom \O_C(n x)$ and $\O(C_0)|_C \isom \O_C(m x)$.  For $D=C+C_0$, set $V=H^0(X,\O(D))$.  Then $\Gamma_\nu(V)$ is finitely generated, and $X$ degenerates to a toric variety whose normalization is isomorphic to the Hirzebruch surface $\mathbf{F}_n$.
\end{proposition}

\begin{proof}
Note that $D$ is maximal with respect to $C$, and the hypothesis means that $H^0(C,D|_C)$ and $H^0(C,(D-C)|_C)$ both contain maximal divisors with respect to $x$; namely, $(n+m) x$ and $m x$, respectively.  We see that $\nu(V)$ contains the points $(0,0)$, $(1,0)$, $(0,n+m)$, and $(1,m)$, and from the description of $\Delta_{Y_\bullet}(V)$ given in \cite[Figure~1]{lm}, these are the vertices of the Okounkov body.  The claim follows from Proposition~\ref{p:fg-easy}.
\end{proof}

\subsection{A counterexample}\label{ss:counterex}

Fix a flag $Y_\bullet$ on $X$ and a linear system $V\subseteq H^0(X,\cL)$.  Suppose that for all $0\leq r\leq d-1$ and all $\ba\in\nu|_r(V)$, the restricted system a linear system $V(\ba)$ contains a maximal divisor with respect to $Y_{r+1}$.  It may still happen that $\nu(V^2) \supsetneq 2\cdot\nu(V)$.

A simple example may be constructed as follows.  On $X=\P^1\times\P^1$, take the flag given by $Y_1 = \{0\}\times\P^1$ and $Y_2 = (0,0)$.  Choose coordinates $\{x,y\}$ around the point $(0,0)$, so $Y_1$ has local equation $\{x=0\}$, let $V \subseteq H^0(X,\O(1,3))$ be the four-dimensional subspace spanned by
\[
  1,\, x,\, y + xy^3,\, xy.
\]
(The image of $X$ in $\P(V)=\P^3$ is a singular cubic surface.)  One easily checks that $\nu(V)$ and $\nu(V^2)$ are as in Figure~\ref{f:counterex}.  In particular, using the notation of \S\ref{s:conditions}, note that 
\[
 \frac{x\cdot (y+xy^3) - 1\cdot xy}{x^2} = y^3
\]
lies in $(V^2)(2)$, but does not come from $V(1)^2$, so the condition $\cC_2(V,V)$ fails.

\begin{figure}
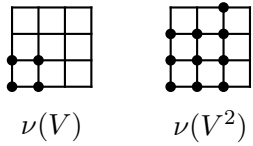

\pspicture(-70,-20)(200,50)
\psset{xunit=10pt,yunit=10pt}
\psgrid[subgriddiv=0,labels=none,gridlabelcolor=white](0,0)(3,3)
\pscircle*(0,0){2}
\pscircle*(0,1){2}
\pscircle*(1,0){2}
\pscircle*(1,1){2}

\rput(1.5,-1.5){$\nu(V)$}

\psgrid[subgriddiv=0,labels=none,gridlabelcolor=white](6,0)(9,3)
\pscircle*(6,0){2}
\pscircle*(6,1){2}
\pscircle*(6,2){2}
\pscircle*(7,0){2}
\pscircle*(7,1){2}
\pscircle*(7,2){2}
\pscircle*(8,0){2}
\pscircle*(8,1){2}
\pscircle*(8,2){2}
\pscircle*(8,3){2}

\rput(7.5,-1.5){$\nu(V^2)$}

\endpspicture

\caption{\small Images of $\nu$ for the linear system $V$ of \S\ref{ss:counterex}.}
\label{f:counterex}
\end{figure}

\subsection{Bott-Samelson varieties}\label{ss:bs}

This example is intended for readers familiar with flag varieties and Bott-Samelson resolutions.

Let $B\subseteq GL_3$ be the subgroup of upper-triangular matrices, and let $T$ be the 
diagonal torus.  Let $\alpha=(1,-1,0)$ and $\beta=(0,1,-1)$ be the two simple roots.  
Set $\ul{\alpha} = (\alpha,\beta,\alpha)$, so we have a Bott-Samelson variety 
$X=X(\ul{\alpha}) = P_\alpha\times^B P_\beta\times^B P_\alpha/B$, together with a birational 
map $\phi:X\to GL_3/B$.  In matrix coordinates, this map is
\begin{equation*}
\left[ \left(\begin{array}{ccc} x & 1 &  \\ 1 &  &  \\ &  & 1\end{array}\right), \;
\left(\begin{array}{ccc}1 &  &  \\ & y & 1 \\ & 1 & \end{array}\right), \;
\left(\begin{array}{ccc} z & 1 &  \\ 1 &  &  \\ &  & 1\end{array}\right) \right]
\mapsto
\left(\begin{array}{ccc} xz+y & x & 1 \\ z & 1 &  \\ 1 &  & \end{array}\right).
\end{equation*}

The flag variety $GL_3/B$ embeds in $\P^2\times\P^2$, and the restriction of $\O(1,1)$ 
is a very ample line bundle $\cL=\cL(\rho)$ on $GL_3/B$.  Its sections correspond 
to semistandard Young tableaux on the partition $(2,1)$; they are products of minors 
of the $3\times 3$ matrix, with the rows of each minor indexed by a column of the tableau.  
To get sections of $\phi^*\cL$, just take these minors of the matrix on the RHS above.  
For example, the tableau 
\pspicture(0,2)(25,25)\psline{-}(0,0)(10,0)(10,20)(0,20)(0,0) 
\psline(0,10)(20,10)(20,20)(10,20) \rput(5,15){$1$}\rput(5,5){$2$}\rput(15,15){$3$} 
\endpspicture
gives $[(xz+y)-xz]\cdot 1 = y$.  The space $U=H^0(X,\phi^*\cL)$ is eight-dimensional, with 
a basis given by the sections
\[
  1,\; x,\; y,\; z,\; xz,\; yz,\; x(xz+y),\; y(xz+y).
\]
Using the above matrix coordinates, take the flag $Y_\bullet$ given by $Y_1 = \{x=0\}$, $Y_2=\{x=y=0\}$, $Y_3=\{x=y=z=0\}$.  

The corresponding valuation $\nu$ evaluates as follows on the sections of $\phi^*\cL$:
\[
\begin{array}{|c|c||c|c|} \hline
 s      &  \nu(s)  &   s    & \nu(s) \\ \hline \hline
1      & (0,0,0)  &   xz & (1,0,1) \\ \hline
x       & (1,0,0)  & yz    & (0,1,1) \\ \hline
y       & (0,1,0)  & x(xz+y)  & (1,1,0) \\ \hline
z     & (0,0,1)   & y(xz+y)   & (0,2,0). \\ \hline
\end{array}
\]

\begin{figure}[t]
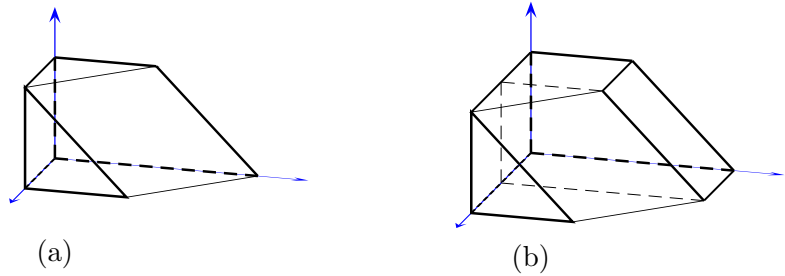

\pspicture(-50,-40)(100,80)

\psset{viewpoint=3.4 1 1.1 } 
\psset{unit=2pt}

\ThreeDput[normal=0 0 1]{
 \psline[linewidth=.5pt,linecolor=blue]{->}(0,0)(30,0)%
}
\ThreeDput[normal=0 0 1]{
 \psline[linewidth=.5pt,linecolor=blue]{->}(0,0)(0,50)%
}

\ThreeDput[normal=1 0 0]{
 \psline[linewidth=.5pt,linecolor=blue]{->}(0,0)(0,30)%
}

\ThreeDput[normal=1 0 0]{
 \psline[linewidth=1pt,linestyle=dashed](0,0)(40,0) 
 \psline[linewidth=1pt,linestyle=dashed](0,0)(0,20)
 \psline[linewidth=1pt](40,0)(20,20)
 \psline[linewidth=1pt](0,20)(20,20)
}

\ThreeDput[normal=1 0 0](20,0,0){
 \pspolygon[linewidth=1pt](0,0)(20,0)(0,20)
}

\ThreeDput[normal=0 0 1]{
 \psline[linewidth=1pt,linestyle=dashed](0,0)(20,0) 
 \psline[linewidth=1pt](0,40)(20,20) 
}

\ThreeDput[normal=0 0 1](0,0,20){
 \psline[linewidth=1pt](0,0)(0,20)(20,0)(0,0) 
}

\rput(0,-18){(a)}

\endpspicture 
\pspicture(-80,-42)(100,80)

\psset{viewpoint=3.4 1 1.1 } 
\psset{unit=2pt}

\ThreeDput[normal=0 0 1]{
 \psline[linewidth=.5pt,linecolor=blue]{->}(0,0)(50,0)%
}
\ThreeDput[normal=0 0 1]{
 \psline[linewidth=.5pt,linecolor=blue]{->}(0,0)(0,50)%
}

\ThreeDput[normal=1 0 0]{
 \psline[linewidth=.5pt,linecolor=blue]{->}(0,0)(0,30)%
}

\ThreeDput[normal=1 0 0]{
 \psline[linewidth=1pt,linestyle=dashed](0,0)(40,0) 
 \psline[linewidth=1pt,linestyle=dashed](0,0)(0,20)
 \psline[linewidth=1pt](40,0)(20,20)
 \psline[linewidth=1pt](0,20)(20,20)
}

\ThreeDput[normal=1 0 0](40,0,0){
 \pspolygon[linewidth=1pt](0,0)(20,0)(0,20)
}

\ThreeDput[normal=1 0 0](20,0,0){
 \psline[linewidth=1pt](40,0)(20,20)
  \psline[linewidth=.2pt,linestyle=dashed](40,0)(0,0)
  \psline[linewidth=.2pt,linestyle=dashed](20,20)(0,20)
  \psline[linewidth=.2pt,linestyle=dashed](0,0)(0,20)
}

\ThreeDput[normal=0 0 1]{
 \psline[linewidth=1pt,linestyle=dashed](0,0)(40,0)
 \psline[linewidth=1pt](40,20)(20,40)
 \psline[linewidth=1pt](20,40)(0,40) 
}

\ThreeDput[normal=0 0 1](0,0,20){
 \psline[linewidth=1pt](0,0)(40,0)(20,20)(0,20) 
}

\rput(0,-20){(b)}

\endpspicture 
\caption{Okounkov bodies for $GL_3/B$ and $X(\alpha,\beta,\alpha)$ \label{f:bs-ex}}
\end{figure}

The Okounkov body $\Delta_{Y_\bullet}(U)$ is the convex hull of these eight points; see 
Firgure~\ref{f:bs-ex}(a).  This is the same polytope as the one described in \cite[Example 6.1]{gonzalez}.  Indeed, 
$GL_3/B$ is isomorphic to $\P(T_{\P^2})$, the projectived tangent bundle of $\P^2$.  Up 
to a lattice isomorphism, the polytope may also be identified with the Gelfand-Tsetlin 
polytope, and it was shown in \cite{km} that the flag varieties for $GL_n$ degenerate to 
these toric varieties.

Finally, let $\cM = \phi^*\cL \otimes pr_1^*\O(1)$, where $pr_1:X \to P_\alpha/B\isom\P^1$ 
is the projection.  One can check that $\cM$ is very ample on $X$, and that $V=H^0(X,\cM)$ 
has a basis consisting of the eight sections spanning $U$, together with the five additional 
ones obtained by multiplying by $x$.  The corresponding Okounkov body is shown in Figure~
\ref{f:bs-ex}(b).  Note that this is isomorphic to the polytope appearing in 
\cite[Figure~2]{fk}.

Recently, Kaveh has given a different computation of Okounkov bodies on Bott-Samelson varieties, 
identifying valuation vectors with lattice points in the \emph{string cone} and establishing a 
connection with the crystal basis \cite{kaveh-crystal}.  Note that the flag used there is different 
from our $Y_\bullet$.




\begin{thebibliography}{KLM}

\bibitem[AB]{ab} Valery~Alexeev and Michel~Brion, ``Toric degenerations of spherical varieties,'' {\em Selecta Math. (N.S.)} {\bf 10} (2004), no. 4, 453--478.

\bibitem[An]{and} Dave~Anderson, ``Okounkov bodies of Bott-Samelson varieties,'' in preparation.

\bibitem[Ba]{bayer} David~Bayer, {\em The Division Algorithm and the Hilbert Scheme}, Ph.D. Thesis, Harvard University, 1982.

\bibitem[BG]{bg} Winfried~Bruns and Joseph~Gubeladze, {\em Polytopes, rings, and $K$-theory}, 
Springer, Dordrecht, 2009.

\bibitem[BH]{bh} Winfried~Bruns and J\"urgen~Herzog, {\em Cohen-Macaulay Rings}, revised edition, Cambridge University Press, Cambridge, 1998.

\bibitem[Ca]{caldero} Philippe~Caldero, ``Toric degenerations of Schubert varieties,'' {\em Transform. Groups} {\bf 7} (2002), no. 1, 51--60.

\bibitem[CK]{ck} James~B.~Carrell and Alexandre~Kurth, ``Normality of torus orbit closures in $G/P$,'' {\em J. Algebra} {\bf 233} (2000), 122--134.

\bibitem[Ei]{eisenbud} David~Eisenbud, {\em Commutative Algebra, with a View Toward Algebraic Geometry}, Springer, 1995.

\bibitem[FK]{fk} Philip~Foth and Sangjib~Kim, ``Toric degenerations of Bott-Samelson varieties,'' preprint, \texttt{arXiv:0905.1374v1 [math.AG]}.

\bibitem[Fu]{fulton} William~Fulton, {\em Introduction to Toric Varieties}, Princeton, 1993.

\bibitem[GL]{gl} N.~Gonciulea and V.~Lakshmibai, ``Degenerations of flag and Schubert varieties to toric varieties,'' {\em Transform. Groups} {\bf 1} (1996), no. 3, 215--248.

\bibitem[Go]{gonzalez} Jos\'e~Luis~Gonz\'alez, ``Okounkov bodies on projectivizations of rank two toric vector bundles,'' {\em J. Algebra} {\bf 330} (2011), 322--345.

\bibitem[Ha]{ha} Robin~Hartshorne, {\em Algebraic Geometry}, Springer, 1977.

\bibitem[Jo]{jow} Shin-Yao~Jow, ``Okounkov bodies and restricted volumes along very general curves,'' {\em Adv. Math.} {\bf 223} (2010), no. 4, 1356--1371.

\bibitem[Ka1]{kaveh} Kiumars~Kaveh, ``SAGBI bases and degeneration of spherical varieties to toric varieties,'' {\em Michigan Math. J.} {\bf 53} (2005), no. 1, 109--121.

\bibitem[Ka2]{kaveh-crystal} Kiumars~Kaveh, ``Crystal bases and Newton-Okounkov bodies,'' preprint, \texttt{arXiv:1101.1687v1 [math.AG]}.

\bibitem[KK1]{kk-alg} Kiumars~Kaveh and A.~G.~Khovanskii, ``Convex bodies and algebraic equations on affine varieties,'' preprint, \texttt{arXiv:0804.4095v1 [math.AG]}.

\bibitem[KK2]{kk} Kiumars~Kaveh and A.~G.~Khovanskii, ``Newton convex bodies, semigroups of integral points, graded algebras and intersection theory,'' to appear in {\em Ann. Math.}  Preprint available as \texttt{arXiv:0904.3350v3 [math.AG]}.

\bibitem[KM]{km} Mikhail~Kogan and Ezra~Miller, ``Toric degeneration of Schubert varieties and Gelfand-Tsetlin polytopes,'' {\em Adv. Math.} {\bf 193} (2005), no. 1, 1--17.

\bibitem[KLM]{klm} Alex~K\"uronya, Victor~Lozovanu, and Catriona~Maclean, ``Convex bodies appearing as Okounkov bodies of divisors,''  {\em Adv. Math.}~{\bf 229} (2012), no. 5, 2622--2639.


\bibitem[La]{rlaz} Robert~Lazarsfeld, {\em Positivity in algebraic geometry I: Classical setting: line bundles and linear series}, Springer-Verlag, 2004.

\bibitem[LM]{lm} Robert~Lazarsfeld and Mircea~Musta{\c{t}}{\u{a}}, ``Convex bodies associated to linear series,'' {\em Ann. Sci. \'Ecole Norm. Sup.} {\bf 42} (2009), no. 5, 783--835.

\bibitem[Ok1]{ok1} Andrei~Okounkov, ``Brunn-Minkowski inequality for multiplicities,'' {\em Invent. Math.} {\bf 125} (1996), no. 3, 405--411.

\bibitem[Ok2]{okounkov} Andrei~Okounkov, ``Multiplicities and Newton polytopes,'' {\em Kirillov's seminar on representation theory}, 231--244, in {\em Amer. Math. Soc. Transl. Ser. 2}, {\bf 181}, Amer. Math. Soc., Providence, RI, 1998.

\bibitem[Te]{teissier} Bernard~Teissier, ``Valuations, deformations, and toric geometry,'' in {\em Valuation theory and its applications, Vol. II} (Saskatoon, SK, 1999), 361--459, Fields Inst. Commun., {\bf 33}, Amer. Math. Soc., Providence, RI, 2003.

\end{thebibliography}
\end{document}